\documentclass[a4paper,centertags,oneside,12pt]{amsart}
\usepackage [latin1]{inputenc}
\usepackage{mathrsfs}
\usepackage{appendix}
\usepackage{amssymb}
\usepackage{fancyhdr}
\usepackage{charter}
\usepackage{typearea}
\usepackage{pdfsync}
\usepackage[colorlinks,linkcolor=blue,anchorcolor=blue,citecolor=green]{hyperref}
\usepackage{mathrsfs}
\usepackage[a4paper,top=3cm,bottom=3cm,left=3cm,right=3cm]{geometry}

\usepackage{enumitem}
\usepackage{color}
\usepackage{cite}

\setlist[1]{itemsep=5pt}
\newcommand{\comment}[1]{}

\makeatletter
     \def\@setcopyright{}
     \def\serieslogo@{}
      \makeatother

\newtheorem{theorem}{Theorem}[section]
\newtheorem{lemma}[theorem]{Lemma}

\newtheorem{proposition}[theorem]{Proposition}

\newtheorem{conjecture}[theorem]{Conjecture}

\numberwithin{equation}{section}

\begin{document}
\title{Rigidity of proper holomorphic self-mappings of the hexablock}
\author{Enchao Bi}
\address{School of Mathematics and Statistics, Qingdao University, Qingdao, Shandong 266071, P.R. China}

\thanks{ }
\email{bienchao@qdu.edu.cn}

\author{Zeinab Shaaban}
\address{School of Mathematics and Statistics, Wuhan University of Technology, Wuhan, Hubei 430070, P.
R. China}

\email{zenab@whut.edu.cn}

\author{Guicong Su$^{*}$}
\address{School of Mathematics and Statistics, Wuhan University of Technology, Wuhan, Hubei 430070, P.
R. China}
\thanks{* Corresponding author.}
\thanks{ }
\email{suguicong@whut.edu.cn}
\begin{abstract}

The hexablock \(\mathbb{H}\), introduced by Biswas-Pal-Tomar \cite{Hexablock}, is a Hartogs domain in \(\mathbb{C}^4\) fibered over the tetrablock \(\mathbb{E}\) in \(\mathbb{C}^3\), arising in the context of \(\mu\)-synthesis problems. In this paper, we prove that every proper holomorphic self-map of \(\mathbb{H}\) is necessarily an automorphism. Consequently, we resolve the conjecture \(G(\mathbb{H}) = \mathrm{Aut}(\mathbb{H})\) on the automorphism group structure, originally posed by Biswas-Pal-Tomar in \cite{Hexablock}.

\end{abstract}
\maketitle

\section{Introduction}

In control theory (see \cite{Do, Fra}), the structured singular value of an $n\times n$ matrix $A$ is a cost function that generalizes the operator norm while encoding structural information about admissible perturbations. This value is defined relative to a prescribed linear subspace $E$ of $M_n(\mathbb{C})$-the space of all $n\times n$ complex matrices as follows
\[
\mu_{E}(A) = \frac{1}{\inf\left\{\|X\| : X \in E,\ \det(I - AX) = 0\right\}}
\]
where $\|\cdot\|$ denotes the operator norm of a matrix relative
to the Euclidean norm  on $\mathbb{C}^n$, with the convention that $\mu_E(A) = 0$ if $\det(I - AX) \neq 0$ for all $X \in E$.

The associated $\mu$-synthesis interpolation problem has revealed deep connections with three distinguished domains in complex geometry:
\begin{itemize}
\item The \textit{symmetrized bidisc} $\mathbb{G}_2 \subset \mathbb{C}^2$ for $E = \{\alpha I : \alpha \in \mathbb{C}\}$.
\item The \textit{tetrablock} $\mathbb{E} \subset \mathbb{C}^3$ for $E = \operatorname{diag}_2(\mathbb{C})$ (the space of $2\times 2$ diagonal matrices).
\item The \textit{pentablock} $\mathbb{P} \subset \mathbb{C}^3$ for $E = \left\{\begin{pmatrix} z & w \\ 0 & z \end{pmatrix} : z,w \in \mathbb{C}\right\}$.
\end{itemize}
These domains exhibit exceptionally rich complex-analytic and operator-theoretic properties, stimulating substantial research across several complex variables, functional analysis, and geometric function theory. Their intricate boundary stratifications, automorphism groups, and invariant metrics provide a foundational framework for addressing fundamental questions in complex geometry. Comprehensive treatments of their geometric invariants and connections to $\mu$-synthesis are provided in \cite{Abouhajar,Abouhajar-White-Young,3,4,11,14,21,22,25,56,57,Kosi,59,60,78,79,80,85,86,89,Young,94} and related works.

%Throughout this paper, $\mathbb{C}$ denotes the set of complex numbers; $\mathbb{D}$ represents the open unit disk in the complex plane, and $\mathbb{T}$ the unit circle. For any subset $\Omega$ of a topological space, $\mathrm{int}(\Omega)$ and $\partial\Omega$ denote its interior and topological boundary, respectively.

Let $E$ be the linear subspace consisting of all $2\times 2$ upper triangular complex matrices, and denote the cost function $\mu_E$ in this case by $\mu_{\text{hexa}}$. We introduce the following mapping
$$
\pi: M_2(\mathbb{C}) \to \mathbb{C}^4, \quad A = \begin{pmatrix} a_{11} & a_{12} \\ a_{21} & a_{22} \end{pmatrix} \mapsto (a_{21}, a_{11}, a_{22}, \det A).
$$
Define the associated set  
$$
\mathbb{H}_\mu = \left\{ \pi(A) : A \in M_2(\mathbb{C}),  \mu_{\text{hexa}}(A) < 1 \right\}.
$$

The set $\mathbb{H}_\mu$ is called the $\mu$-hexablock by Biswas-Pal-Tomar \cite{Hexablock}. However, $\mathbb{H}_\mu$ is connected but not open in $\mathbb{C}^4$, and therefore not a domain in $\mathbb{C}^4$. In particular, points of the form $(0,x_1,x_2,x_3)$ are not interior points of $\mathbb{H}_\mu$.

Let $\mathbb{E}$ be the tetrablock in $\mathbb{C}^3$ (refer to \cite{Abouhajar, Abouhajar-White-Young}), which is given by 
$$\mathbb{E} = \left\{ (x_1, x_2, x_3) \in \mathbb{C}^3 : 1 - x_1 z - x_2 w + x_3 z w \neq 0 \quad \forall\;  |z| \leq 1, \forall \;|w| \leq 1 \right\}.$$
We consider a new family of fractional linear maps defined by
\begin{equation}\label{linear}
   \Psi_{z_1,z_2}(a, x_1, x_2, x_3) := \frac{a \sqrt{(1 - |z_1|^2)(1 - |z_2|^2)}}{1 - x_1 z_1 - x_2 z_2 + x_3 z_1 z_2}, 
\end{equation}
where $z_1,z_2\in\overline{\mathbb{D}}$, and $(x_1, x_2, x_3) \in \mathbb{E}$. 

Define  
\[
\mathbb{H} := \left\{ (a, x_1, x_2, x_3) \in \mathbb{C} \times \mathbb{E} : \sup_{z_1,z_2 \in \mathbb{D}} \left| \Psi_{z_1,z_2}(a, x_1, x_2, x_3) \right| < 1 \right\}.
\]
This set $\mathbb{H}$, introduced by Biswas-Pal-Tomar \cite{Hexablock}, is named the \textit{hexablock} owing to the geometric structure of its real slice $\mathbb{H} \cap \mathbb{R}^4$: a bounded set whose boundary consists of precisely four extreme sets and two hypersurfaces.

Unlike the $\mu$-hexablock $\mathbb{H}_\mu$ (which fails to be a domain in $\mathbb{C}^4$), the hexablock $\mathbb{H}$ is a domain in $\mathbb{C}^4$. Moreover, it satisfies the fundamental relation:
\[
\mathbb{H} = \operatorname{int}(\overline{\mathbb{H}}_\mu).
\]
Furthermore, $\mathbb{H}$ is linearly convex, yet neither circular nor convex, and its boundary is not $\mathcal{C}^1$-smooth (see Biswas-Pal-Tomar \cite{Hexablock}).

By its construction as a Hartogs domain over the tetrablock $\mathbb{E} \subset \mathbb{C}^3$, the hexablock $\mathbb{H}$ admits the following characterization: for each $x = (x_1,x_2,x_3) \in \mathbb{E}$, there exists a unique point $(z_1(x), z_2(x)) \in \mathbb{D}^2$ at which the function
\[
(z_1, z_2) \mapsto \left| \frac{1 - x_1 z_1 - x_2 z_2 + x_3 z_1 z_2}{\sqrt{(1 - |z_1|^2)(1 - |z_2|^2)}} \right|^{-1}
\]
attains its supremum over $\mathbb{D}^2$ (see Chapter 3 in Biswas-Pal-Tomar \cite{Hexablock}). Consequently, $\mathbb{H}$ has the explicit representation:
\[
\mathbb{H} = \left\{ (a, x) \in \mathbb{C} \times \mathbb{E} : |a|^2 < e^{-u(x)} \right\},
\]
where $x = (x_1, x_2, x_3)$ and the smooth function $u: \mathbb{E} \to \mathbb{R}^+$ is given by
\begin{equation}\label{u}
u(x) = 2 \log \left| \frac{\sqrt{(1 - |z_1(x)|^2)(1 - |z_2(x)|^2)}}{1 - x_1 z_1(x) - x_2 z_2(x) + x_3 z_1(x) z_2(x)} \right|.
\end{equation}
Recently, Biswas-Pal-Tomar \cite{Hexablock} constructed a special subgroup of the automorphism group of the hexablock. More precisely, they devised a subgroup of $\operatorname{Aut}(\mathbb{H})$ that preserves the tetrablock $\mathbb{E}.$
\begin{theorem}[Theorem 7.10 in \cite{Hexablock}]
  The collection
  \[
  G(\mathbb{H}) = \left\{ T_{\nu,\chi,\omega},\  T_{\nu,\chi,F,\omega} : \nu,\chi \in \operatorname{Aut}(\mathbb{D}),\  \omega \in \mathbb{T} \right\}
  \]
  forms a subgroup of $\operatorname{Aut}(\mathbb{H})$.
\end{theorem}

The automorphisms are explicitly given by:
\begin{align}
T_{\nu,\chi,\omega}(a,x) &= \left( \omega\xi_2 \frac{a\sqrt{(1-|z_1|^2)(1-|z_2|^2)}}{1-x_1\overline{z}_1 - x_2\overline{z}_2\xi_2 + x_3\overline{z}_1\overline{z}_2\xi_2},\  \tau_{\nu,\chi}(x) \right), \label{form1} \\
T_{\nu,\chi,F,\omega}(a,x) &= \left( \omega\xi_2 \frac{a\sqrt{(1-|z_1|^2)(1-|z_2|^2)}}{1-x_1\overline{z}_1 - x_2\overline{z}_2\xi_2 + x_3\overline{z}_1\overline{z}_2\xi_2},\  \tau_{\nu,\chi,F}(x) \right), \label{form2}
\end{align}
where $x = (x_1,x_2,x_3)$ and $\tau_{\nu,\chi}, \tau_{\nu,\chi,F} \in \operatorname{Aut}(\mathbb{E})$ are the tetrablock automorphisms characterized by Young \cite{Young}. Here $\nu = -\xi_1B_{z_1}$, $\chi = -\xi_2B_{-\overline{z}_2}$ with Blaschke factors
\[
B_\alpha(\lambda) = \frac{\lambda - \alpha}{\overline{\alpha}\lambda - 1}, \quad \alpha \in \mathbb{D},
\]
for parameters $\xi_1,\xi_2 \in \mathbb{T}$ and $z_1,z_2 \in \mathbb{D}$.

In their foundational work on the hexablock, Biswas-Pal-Tomar \cite{Hexablock} formulated the fundamental rigidity conjecture:
\begin{conjecture} 
$G(\mathbb{H}) = \mathrm{Aut}(\mathbb{H})$.
\end{conjecture}

Our resolution of this conjecture is achieved through a complete characterization of proper holomorphic self-mappings of $\mathbb{H}$. Specifically, we prove:

\begin{theorem}\label{Main}
Every proper holomorphic self-mapping of $\mathbb{H}$ is an automorphism, and moreover, coincides with one of the explicit forms given by \eqref{form1} or \eqref{form2}. This characterization directly establishes $G(\mathbb{H}) = \mathrm{Aut}(\mathbb{H})$.
\end{theorem}

\section{Preliminaries}
\subsection{Tetrablock}
%The tetrablock is the domain $%\mathbb{E}\subset\mathbb{C}^3$ given by
%$$\mathbb{E} = \left\{ (x_1, x_2, x_3) \in \mathbb{C}^3 : 1 - x_1 z - x_2 w + x_3 z w \neq 0 \quad \forall  |z| \leq 1, \forall |w| \leq 1 \right\}.$$
In this section, we review the fundamental properties of the tetrablock $\mathbb{E}$ (see \cite{Young} for details). This domain admits several equivalent characterizations.
\begin{lemma}[Theorem 2.2 in \cite{Abouhajar-White-Young}]
For $x=(x_1,x_2,x_3)\in\mathbb{C}^3$, the following are equivalent:
\begin{enumerate}
    \item $x\in\mathbb{E}$;
    \item $|x_1-\bar{x}_2x_3|+|x_1x_2-x_3|<1-|x_2|^2$;
    \item $\sup_{z\in\mathbb{D}}|\Psi(z,x)|<1$ and if $x_1x_2=x_3$ then $|x_2|<1$, where
          \[
          \Psi(z,x_1,x_2,x_3)=\frac{x_3z-x_1}{x_2z-1};
          \]
    \item There exists a symmetric matrix $A\in\mathbb{C}^{2\times2}$ with $\|A\|<1$ such that $\pi(A)=x$;
    \item There exist $\beta_1,\beta_2\in\mathbb{C}$ with $|\beta_1|+|\beta_2|<1$ satisfying
          \[
          x_1=\beta_1+\bar{\beta}_2x_3,\quad x_2=\beta_2+\bar{\beta}_1x_3.
          \]
\end{enumerate}
\end{lemma}

Importantly, the automorphism group of $\mathbb{E}$ acts on the $\beta$-foliation. This connection reveals fundamental aspects of the domain's geometry.

\begin{theorem}[Theorem 5.2 in \cite{Young}]\label{Young}
Let $x=(x_1,x_2,x_3) \in \mathbb{E}$. The orbit of $x$ under $\text{Aut}(\mathbb{E})$ contains a unique point of the form $(0,0,r)$ with $r \in [0,1)$. If $x = (\beta_1 + \bar{\beta}_2 \lambda, \beta_2 + \bar{\beta}_1 \lambda, \lambda)$ then $r$ is given by
$$r = \left| \frac{\lambda - \alpha \bar{\theta}}{\bar{\alpha} \theta \lambda - 1} \right|,$$
where $\alpha$ and $\theta$ are explicitly determined by $\beta_1$ and $\beta_2$.
\end{theorem}

We establish the following analytic characterizations of the topological boundary $\partial\mathbb{E}$.

\begin{lemma}[Corollary 4.2.7 in \cite{Abouhajar}]\label{Le2.3}
For $(x_1,x_2,x_3)\in \mathbb{C}^3$, the following are equivalent:
\begin{enumerate}
    \item $(x_{1},x_{2},x_{3})\in\partial\mathbb{E}$;
    \item $|x_{1}|\leq1$ and $|x_{2}|^{2}+|x_{1}-\overline{x}_{2}x_{3}|+|x_{1}x_{2}-x_{3}|=1$;
    \item $|x_{2}|\leq1$ and $|x_{1}|^{2}+|x_{2}-\overline{x}_{1}x_{3}|+|x_{1}x_{2}-x_{3}|=1$;
    \item $|x_{1}|\leq1$, $|x_{2}|\leq1$, $|x_{3}|\leq1$ and $1-|x_{1}|^{2}-|x_{2}|^{2}+|x_{3}|^2-2|x_1x_2-x_3|=0$;
    \item $|x_{1}|\leq1$, $|x_{2}|\leq1$ and $|x_{1}-\overline{x}_{2}x_{3}|+|x_{2}-\overline{x}_{1}x_{3}|=1-|x_{3}|^{2}$.
\end{enumerate}
\end{lemma}

Regarding boundary structures, recall that for a domain $\Omega\subset\mathbb{C}^n$, 
its distinguished boundary $b\Omega$ is defined as the smallest closed subset of $\overline{\Omega}$ 
where every function holomorphic on $\Omega$ and continuous on $\overline{\Omega}$ attains its maximum modulus. The distinguished boundary of the tetrablock was fully characterized by Abouhajar-White-Young \cite{Abouhajar-White-Young}.

\begin{lemma}[Theorem 7.1 in \cite{Abouhajar-White-Young}]\label{le2.5}
For $x=(x_1,x_2,x_3)\in\mathbb{C}^3$, the following are equivalent:
\begin{enumerate}
    \item $x_1=\overline{x_2}x_3$, $|x_3|=1$ and $|x_2|\leq 1$;
    \item Either $x_1x_2\neq x_3$ and $\Psi(\cdot,x_1,x_2,x_3)\in\mathrm{Aut}(\mathbb{D})$, or $x_1x_2=x_3$ and $|x_1|=|x_2|=|x_3|=1$;
    \item There exists a $2\times 2$ unitary matrix $U$ such that $x=(u_{11},u_{22},\det U)$;
    \item $x\in b\mathbb{E}$;
    \item $x\in\overline{\mathbb{E}}$ and $|x_3|=1$.
\end{enumerate}
\end{lemma}

\subsection{Hexablock}
As mentioned before, the hexablock is a Hartogs domain fibered over the tetrablock $\mathbb{E}$. In the expression of $u(x)$, $z_1(x)$ and $z_2(x)$ can be explicitly given by (see \cite{Hexablock})

\begin{equation*}
z_{1}(x)=\frac{2\overline{\beta}_{1}}{1+|\beta_{1}|^{2}-|\beta_{2}|^{2}+\sqrt{(1+|\beta_{1}|^{2}-|\beta_{2}|^{2})^{2}-4|\beta_{1}|^{2}}}
\end{equation*}
and
\begin{equation*}
    z_{2}(x)=\frac{2\overline{\beta}_2}{1+|\beta_2|^2-|\beta_1|^2+\sqrt{(1+|\beta_2|^2-|\beta_1|^2)^2-4|\beta_2|^2}},
\end{equation*}
where
\begin{equation*}
\beta_1=\frac{x_1-\overline{x}_2x_3}{1-|x_3|^2}\quad and\quad\beta_2=\frac{x_2-\overline{x}_1x_3}{1-|x_3|^2}.
\end{equation*}

Then by a straightforward computation, it is not hard to see $\overline{\mathbb{H}}\subset \overline{\mathbb{B}^2}\times \overline{\mathbb{D}}\times \overline{\mathbb{D}}$. More precisely, we have the following result.
\begin{lemma}[Lemma 6.5 in \cite{Hexablock}]\label{le2.3}
Let $(a,x_1,x_2,x_3)\in \overline{\mathbb{H}}$. Then $|a|^2+|x_1|^2\leq1$ and $|a|^2+|x_2|^2\leq1$. Moreover, if $(a,x_1,x_2,x_3)\in\mathbb{H}$, then $|a|^2+|x_1|^2<1.$
\end{lemma}
In addition, we also need some useful geometric properties of the hexablock.
\begin{proposition}[Proposition 6.9 and 6.10 in \cite{Hexablock}]
The hexablock is a polynomially convex and $(1,1,1,2)-$quasi-balanced domain.
\end{proposition}
We recall that a domain $\Omega\subset\mathbb{C}^n$ is said to be $(m_1,\cdots,m_n)$-circular (shortly quasi-circular) if
$$\begin{pmatrix}\lambda^{m_1}z_1,\cdots,\lambda^{m_n}z_n\end{pmatrix}\in \Omega$$
for all $\lambda\in\mathbb{T}$ and $z=(z_1,\cdots,z_n)\in \Omega.$ If the relation holds for all $\lambda\in\mathbb{D}$, then $\Omega$ is said to be $(m_1,\cdots,m_n)$-balanced. A very interesting result about $(m_1,\cdots,m_n)$-balanced domains was the holomorphic extension of proper holomorphic mappings between two such domains, which was obtained by Kosi\' {n}ski \cite{Kosi}.
\begin{lemma}[Lemma 6 in \cite{Kosi}]\label{extension}
 Let $D, G$ be bounded domains in $\mathbb{C}^n$. Suppose that $G$ is $( m_1, \cdots , m_n)$-circular and that it contains the origin. Furthermore, we assume that the Bergman kernel function $K_{D}(z,\bar{\xi})~(z,\xi\in D)$ associated with $D$ satisfies the following property: for any open, relatively compact subset $E$ of $D$, there is an open set $U=U(E)$ that contains $\overline{D}$ such that $K_D(z,\bar{\xi})$ extends to be holomorphic on $U$ as a function of $z$ for each $\xi \in E$. Then, any proper holomorphic mapping $f: D \to G$ extends holomorphically to a neighborhood of $\overline {D}.$
\end{lemma}
\begin{theorem}\label{Ex}
Any proper holomorphic self-mapping of the hetrablock $\mathbb{H}$ extends holomorphically to a neighborhood of $\mathbb{H}$.
\end{theorem}
\begin{proof}
By Corollary 6.4 in Biswas-Pal-Tomar \cite{Hexablock}, we know that 
$$\mathbb{H}_r:=\{(ra,rx_1,rx_2,r^2x_3),\;(a,x_1,x_2,x_3)\in \mathbb{H}\}$$
are relatively compact in $\mathbb{H}$ for any $0<r<1$. Then the conclusion follows from Remark 7 in Kosi\'{n}ski 
\cite{Kosi} and Lemma \ref{extension}.
\end{proof}
\subsection{Geometric properties of the topological boundary of the hexablock} 
We shall divide the boundary of hexablock into three parts $\partial\mathbb{H}=\partial_1\mathbb{H}\cup\partial_2\mathbb{H}\cup\partial_3\mathbb{H}$, where
\begin{equation*}
  \begin{aligned}
(1)\; & \partial_1\mathbb{H}=\partial\mathbb{H}\cap(\mathbb{C}\times\mathbb{E}),\\
(2)\; & \partial_2\mathbb{H}=\partial\mathbb{H}\cap(\mathbb{C}\times\left(\partial\mathbb{E}\setminus(b\mathbb{E}\cup\{(x_1,x_2,x_3)\in\overline{\mathbb{E}}: x_1x_2=x_3\}\right),\\
(3)\; & \partial_3\mathbb{H}=\partial\mathbb{H}\setminus{(\partial_1\mathbb{H}\cup\partial_2\mathbb{H})}.
  \end{aligned}  
\end{equation*} 
Note that our division differs from that in Biswas-Pal-Tomar \cite{Hexablock}. The boundary properties under this new division will be analyzed.

\textbf{I.} It is evident that 

$$\partial_1\mathbb{H}=\left\{(a,x_{1},x_{2},x_{3})\in\mathbb{C}\times\mathbb{E}:|a|^{2}=e^{-u(x_{1},x_{2},x_{3})}\right\}.$$

The topological codimension of $\partial_1\mathbb{H}$ is equal to $1$ and any point $p$ of $\partial_1\mathbb{H}$ is a smooth
point of $\partial\mathbb{H}$. In addition, we will show the following.

\begin{lemma}\label{b1}
There are no $2$-dimensional analytic discs in $\partial_1\mathbb{H}$, and $\partial_1\mathbb{H}$ may be foliated with $1$-dimensional analytic discs.
\end{lemma}

\begin{proof}
Assume there exists a $2$-dimensional analytic disc
\[
\varphi(z_1,z_2) = \left(f(z_1,z_2), g_1(z_1,z_2), g_2(z_1,z_2), g_3(z_1,z_2)\right) : \mathbb{D}^2 \to \partial_1\mathbb{H}
\]
with $\varphi(0,0) = (a^0, x^0_1, x^0_2, x^0_3) \in \partial_1\mathbb{H}$. By Theorem \ref{Young}, the automorphisms \eqref{form1} and \eqref{form2}, we may assume
\[
(x^0_1, x^0_2, x^0_3) = (0, 0, r), \quad 0 \leq r < 1.
\]
This implies $z_1(x^0) = z_2(x^0) = 0$ (the values maximizing $|\Psi_{z_1,z_2}|$ in \eqref{linear}), and consequently $u(x^0_1, x^0_2, x^0_3) = 0$. Therefore, we must have $|a^0| = 1$. 

By Lemma \ref{le2.3}, $f(z_1, z_2)$ attains its maximum at $(0, 0)$, and thus is constant. Applying Lemma \ref{le2.3} again yields
\[
|g_1(z_1, z_2)|^2 = |g_2(z_1, z_2)|^2 = 0 \quad \forall \, (z_1, z_2) \in \mathbb{D}^2,
\]
which implies $g_1(z_1, z_2) \equiv 0$ and $g_2(z_1, z_2) \equiv 0$. Consequently, $\varphi$ reduces to the form
\[
\varphi = \left(e^{i\theta}, 0, 0, g_3(z_1, z_2)\right).
\]

However, this contradicts the 2-dimensionality of $\varphi$ since $\operatorname{dim}_\mathbb{C} \varphi(\mathbb{D}^2) \leq 1$.
\end{proof}

\textbf{II.} Apparently,

\begin{equation*}
    \partial_2\mathbb{H} = \big\{ (a, x_1, x_2, x_3) \in\mathbb{C}\times \partial\mathbb{E}:\ 
    |a|^2 < e^{-u(x_1,x_2,x_3)},\ 
    x_1x_2 \neq x_3,\ 
    |x_3| \neq 1 \big\}.
\end{equation*}
The topological codimension of $\partial_2\mathbb{H}$ is evidently  $1$. In fact, we have the following characterization:

\begin{lemma}
$\partial_2\mathbb{H}:=\left\{(a,x_1,x_2,x_3)\in\mathbb{C}\times\partial\mathbb{E}: |a|^2<|x_1x_2-x_3|, \vert x_3\vert\neq 1\right\}.$
\end{lemma}

\begin{proof}
 Consider a point $(a, x_1, x_2, x_3) \in \partial_2\mathbb{H}$, and let $x := (x_1, x_2, x_3) \in \partial\mathbb{E} \setminus b\mathbb{E}$ with $x_1 x_2 \neq x_3$. By Lemma 9.7 in Biswas-Pal-Tomar \cite{Hexablock}, we may assume that
$$x = (0, r, 1 - r), \quad \text{for some } r \in (0, 1).$$  
According to the definition of $\overline{\mathbb{H}}$, it suffices to compute  
$$\sup_{z_1,z_2\in\mathbb{D}} \left| \frac{(1 - |z_1|^2)(1 - |z_2|^2)}{(1 - r z_2 + (1 - r) z_1 z_2)^2} \right|.$$

Set $z_1 = a e^{i\theta_1}$ and $z_2 = b e^{i\theta_2}$ with $0 \leq a, b < 1$. Then
\begin{align*}
    \left| 1 - r z_2 + (1-r) z_1 z_2 \right| 
    &= \left| 1 - r b e^{i\theta_2} + (1-r) a b e^{i(\theta_1+\theta_2)} \right| \\
    &= \left| 1 - b e^{i\theta_2} \left( r - (1-r) a e^{i\theta_1} \right) \right| \\
    &\geq 1 - b \left| r - (1-r) a e^{i\theta_1} \right|, 
\end{align*}
where 
\[
\left| r - (1-r) a e^{i\theta_1} \right| \leq r + (1-r) a < 1.
\]

Consequently, we obtain the uniform estimate:
$$\left| \frac{(1 - |z_1|^2)(1 - |z_2|^2)}{(1 - r z_2 + (1-r) z_1 z_2)^2} \right| \leq \frac{(1 - a^2)(1 - b^2)}{[1 - b (r + (1-r) a)]^2} =: h(a, b).$$

Hence, 
$$\frac{\partial h}{\partial a} = \frac{2(1 - b^2)}{[1 - b(r + (1-r)a)]^3} \left( (1-r)b - a(1 - br) \right).$$

This implies that the critical point occurs when:
$$(1-r)b - a(1 - br) = 0 \quad \Rightarrow \quad a = \frac{(1-r)b}{1 - br}.$$
At this critical point, we have:
\begin{equation*}
    \begin{aligned}
h(a,b) &\leq h\left( \frac{(1-r)b}{1 - br}, b \right) \\
&= \frac{(1 - b^2) \left[ (1 - br)^2 - (1-r)^2b^2 \right]}{\left[ (1 - br)^2 - (1-r)^2b^2 \right]^2} \\
&= \frac{1 + b}{1 + b - 2br}.
\end{aligned}
\end{equation*}

Consequently, the supremum is given by:
$$\sup_{|z_1| < 1, \, |z_2| < 1} \left| \frac{(1 - |z_1|^2)(1 - |z_2|^2)}{(1 - r z_2 + (1 - r) z_1 z_2)^2} \right| = \frac{1}{1 - r},$$
where we take the limit as $b \to 1^-$. This establishes the boundary condition:
\begin{equation}\label{boundary}
    |a|^2 < 1 - r.
\end{equation}

For any point $(a^0, x^0_1, x^0_2, x^0_3) \in \partial_2\mathbb{H}$, the automorphism \eqref{form1} or \eqref{form2} provides an explicit automorphism $F \in \mathrm{Aut}(\mathbb{H})$ with parameters $z_1 = x^0_1, z_2 = 0$ such that  
$$F(a^0, x^0_1, x^0_2, x^0_3) = \left( w\xi_2 \frac{a^0}{\sqrt{1 - |x^0_1|^2}},\  0,\  \xi_2 \frac{x^0_2 - \overline{x^0_1} x^0_3}{1 - |x^0_1|^2},\  -\xi_1\xi_2 \frac{x^0_1 x^0_2 - x^0_3}{1 - |x^0_1|^2} \right).$$  
Denote this image point by  
$$\left( \widetilde{a},\ 0,\ \widetilde{r},\ 1 - \widetilde{r} \right).$$  
Applying the boundary condition \eqref{boundary} to $\widetilde{a}$ and $\widetilde{r}$, we obtain  
$$|a^0|^2 < |x^0_1 x^0_2 - x^0_3|.$$  
This completes the proof.
\end{proof}

This leads to the fundamental foliation structure.
\begin{lemma}\label{b2}
The boundary component $\partial_2\mathbb{H}$ admits a foliation by $2$-dimensional analytic discs.   
\end{lemma}

\begin{proof}

Define the  functions
$$
x_1(\lambda) := -\frac{1 - r + h(\lambda)}{2 - r}, \quad 
x_2(\lambda) := \frac{1 + (1 - r)h(\lambda)}{2 - r}, \quad 
x_3(\lambda) := -h(\lambda),
$$
where $ h: \mathbb{D} \to \mathbb{D} $ is a holomorphic function with $h(0) = -(1 - r)$ and for some fixed  $0 < r < 1$.

Direct computation establishes that:
\begin{gather*}
(x_1(\lambda), x_2(\lambda), x_3(\lambda)) \in \partial\mathbb{E} \setminus b\mathbb{E}, \\
\left| x_1(\lambda)x_2(\lambda) - x_3(\lambda) \right| = \frac{1 - r}{(2 - r)^2} |1 - h(\lambda)| \neq 0.
\end{gather*}

Given $a > 0$ satisfying  $a^2 < 1 - r$, set $\epsilon := \frac{1 - r - a^2}{2} > 0$. Since \( x_1(\lambda)x_2(\lambda) - x_3(\lambda) \) is holomorphic and non-vanishing, with $|x_1(0)x_2(0) - x_3(0)| = 1 - r$, we may choose $h(\lambda)$ such that
\begin{equation*}
1 - r - \epsilon < \left| x_1(\lambda)x_2(\lambda) - x_3(\lambda) \right| < 1 - r + \epsilon, \quad \forall \lambda \in \mathbb{D}.
\end{equation*}

Correspondingly, there exists a holomorphic function $g: \mathbb{D} \to \mathbb{C}$ with $ g(0) = a $ and 
\begin{equation*}
a^2 - \epsilon < |g(\mu)|^2 < a^2 + \epsilon, \quad \forall \mu \in \mathbb{D}.
\end{equation*}

This implies the strict inequality
\begin{equation*}
|g(\mu)|^2 < \left| x_1(\lambda)x_2(\lambda) - x_3(\lambda) \right|.
\end{equation*}

Consequently, the mapping 
\[
(\mu, \lambda) \mapsto (g(\mu), x_1(\lambda), x_2(\lambda), x_3(\lambda)), \quad \mu, \lambda \in \mathbb{D},
\]
defines a \( 2 \)-dimensional analytic disc embedded in \( \partial_2\mathbb{H} \) that intersects the boundary point \( (a, 0, r, 1 - r) \).
\end{proof}

\textbf{III.} The topological codimension of $\partial_3\mathbb{H}$ is equal to $2$.

\section{Proof of Theorem \ref{Main}}

Let $f=(f_1,f_2,f_3,f_4):\mathbb{H}\rightarrow\mathbb{H}$ be a proper holomorphic mapping. By Theorem \ref{Ex}, we can see that $f$ extends holomorphically
to a neighborhood $V$ of $\mathbb{H}$ with $f(\partial\mathbb{H})\subset\partial\mathbb{H}$. We define
$$S:=\left\{\xi\in V: J_f(\xi)=0\right\},$$
where $J_\varphi(\xi)=\det\left(\frac{\partial f_i}{\partial\xi_j}\right)$ is the complex Jacobian determinant of $f(\xi)$ $(\xi\in V)$.

Consider the mapping $\varphi:\mathbb{E}\rightarrow\mathbb{E}$ defined by 
$$\varphi:(x_1,x_2,x_3)\mapsto\left(f_2(0,x_1,x_2,x_3),f_3(0,x_1,x_2,x_3),f_4(0,x_1,x_2,x_3)\right).$$
Thus, $\varphi$ extends holomorphically to the closure $\overline{\mathbb{E}}$. Next, we will prove that $\varphi:\mathbb{E}\rightarrow\mathbb{E}$ is a proper holomorphic  self-mapping of $\mathbb{E}$.

By Lemma \ref{b2}, for each point  $q \in \partial_2\mathbb{H}$, there exists a $2$-dimensional analytic disk contained entirely within $\partial_2\mathbb{H} $ that passes through the point $q$. However, Lemma \ref{b1} implies that no such $2$-dimensional analytic disk lies in $\partial_1\mathbb{H}$. Consequently, the local biholomorphism of $f$ on $\partial_2\mathbb{H}\setminus S$ ensures that for any point $p \in \partial_2\mathbb{H}\setminus S$, we can find a $2$-dimensional analytic disk that passes through $f(p)$ and is contained within $\partial\mathbb{H}$. It follows that $f(p) \in \partial\mathbb{H} \setminus \partial_1\mathbb{H}\subset\mathbb{C}\times\partial\mathbb{E}$, for all $p \in \partial_2\mathbb{H}\setminus S$.

By Lemma \ref{Le2.3}, we have 
$$\partial\mathbb{E}=\left\{x=(x_1,x_2,x_3)\in\overline{\mathbb{D}}^3: 
|x_1|^2+|x_2|^2-|x_3|^2+2|x_1x_2-x_3|=1\right\}.$$
Thus, for  $p \in \partial_2\mathbb{H}\setminus S$, we have 
\begin{equation*}
|f_2(a,x)|^2+|f_3(a,x)|^2-|f_4(a,x)|^2 +2|f_2(a,x)f_3(a,x)-f_4(a,x)|=1.       
   \end{equation*}

Since $\partial_2\mathbb{H}\setminus S$  is dense in 
 $\partial_2\mathbb{H}$ and  $f$ is continuous on $\overline{\mathbb{H}}$, it follows that

\begin{equation*}
|f_2(a,x)|^2+|f_3(a,x)|^2-|f_4(a,x)|^2+2|f_2(a,x)f_3(a,x)-f_4(a,x)|=1      \end{equation*}
for all $p \in \partial_2\mathbb{H}$.

Consider a point $(0,x)$ where $x=(x_1,x_2,x_3)$ satisfies $x\in\partial\mathbb{E}\setminus b\mathbb{E}$ and $x_1x_2\neq x_3$. Then $(0,x)\in\partial_2\mathbb{H}$, which implies that 
\begin{equation*}
|f_2(0,x)|^2+|f_3(0,x)|^2-|f_4(0,x)|^2+2|f_2(0,x)f_3(0,x)-f_4(0,x)|=1      
\end{equation*}
for all $(0,x)\in\partial_2\mathbb{H}$.
By the continuity of $f$ on $\overline{\mathbb{H}}$,
this extends to all $x\in\partial\mathbb{E}$. Since clearly $|f_k(0,x)|\leq 1$, for $k=2,3,4$ and all $x\in\partial\mathbb{E}$, the map $\varphi$ maps $\partial\mathbb{E}$ onto $\partial\mathbb{E}$. Consequently,
$$\varphi:x\mapsto\left(f_2(0,x),f_3(0,x),f_4(0,x)\right)$$
is a proper holomorphic self-mapping of $\mathbb{E}$.
Therefore, $\varphi$ is an automorphism of $\mathbb{E}$ by Theorem 1 in Kosi\'{n}ski \cite{Kosi}. Furthermore, 
\begin{equation}\label{3.1}
\varphi(b\mathbb{E})\subset b\mathbb{E}.
\end{equation}

Applying Lemma \ref{le2.5} yields the following.
\begin{equation}\label{3.2}
    b\mathbb{E}=\{|x_3|=1, x_1=\bar{x}_2x_3,\ \ \text{and}\ \ |x_2|<1\}.
\end{equation}

Fix $x=(x_1,x_2,x_3)\in b\mathbb{E}$ with $x_3\neq x_1x_2$. This implies $x$ is a non-triangular point, since otherwise the conditions $|x_1|=|x_2|=1$ would force $a=0$ by Lemma \ref{le2.3}. Consequently, for $|a|^2<e^{-u(x)}$, we have $(a,x)\in\partial\mathbb{H}$. However, \eqref{3.1} and \eqref{3.2} imply that
\begin{equation*}
    |f_4(0,x)|=1
\end{equation*}
for all $x\in b\mathbb{E}$. Therefore, for each fixed $x\in b\mathbb{E}$ with $x_3\neq x_1x_2$, the function $f_4(a,x)$ attains its maximum modulus on the disk $\left\{a: |a|^2<e^{-u(x)}\right\}$ at $a=0$. By the maximum modulus principle, $f_4(a,x)$ is thus independent of $a$ at such boundary points.

Continuing to fix $x\in b\mathbb{E}$ with $x_3\neq x_1x_2$, we deduce from \eqref{3.2} that 
\[
f_4(a,x) = e^{i\theta} \quad \text{and} \quad f_2(a,x) = \overline{f_3(a,x)}e^{i\theta}
\]
for some $\theta\in\mathbb{R}$ and all $|a|^2<e^{-u(x)}$. Consequently, both $f_2$ and $f_3$ are independent of the variable $a$ at such boundary points.

Consequently, for every positive integer $k$ and all $x \in b\mathbb{E} \setminus \{x_3 = x_1x_2\}$, the partial derivatives vanish at the point $(0,x)$, i.e.,
\[
\frac{\partial^k f_2}{\partial a^k}(0,x) = 0, \quad
\frac{\partial^k f_3}{\partial a^k}(0,x) = 0, \quad
\frac{\partial^k f_4}{\partial a^k}(0,x) = 0.
\]

By holomorphicity of $f$ on $\overline{\mathbb{H}}$, these vanish identically on $b\mathbb{E}$. Since $b\mathbb{E}$ is the Shilov boundary of $\mathbb{E}$, we extend these equations to the entire domain. Specifically, for all $k \in \mathbb{Z}^+$ and $x \in \mathbb{E}$
\[
\frac{\partial^k f_2}{\partial a^k}(0,x) \equiv 0, \quad
\frac{\partial^k f_3}{\partial a^k}(0,x) \equiv 0, \quad
\frac{\partial^k f_4}{\partial a^k}(0,x) \equiv 0.
\]

Therefore, the functions $f_2(a,x), f_3(a,x)$, and $f_4(a,x)$ defined on $\mathbb{H}$ are independent of $a$. Consequently, the proper holomorphic mapping $f$ admits the following representation
$$f(a,x) = \big( f_1(a,x), f_2(x), f_3(x), f_4(x) \big), \quad x = (x_1,x_2,x_3).$$

Thus, by Theorem 4.1 in Young \cite{Young}, the automorphisms \eqref{form1} and \eqref{form2}, there exists $T\in\textrm{Aut}(\mathbb{H})$  such that 
$$G:=T\circ f(a,x)=(F_1(a,x), x), \quad x=(x_1,x_2,x_3).$$

Subsequently, by Lemma 3.7 in  Biswas-Pal-Tomar \cite{Hexablock}, we define the following subdomain
$$\Omega:=\left\{(a,\lambda,\lambda,\lambda^2)\in\mathbb{H}\right\}
=\left\{(a,\lambda,\lambda,\lambda^2)\in\mathbb{C}^4: |a|^2+|\lambda|^2<1\right\}.$$
It follows that $G(a,x)$ preserves $\Omega$, i.e., $G(\Omega)\subset\Omega$.

Correspondingly, define  
$$\widetilde\Omega:=\left\{(a,\lambda): (a,\lambda,\lambda,\lambda^2)\in\mathbb{H}\right\}
=\left\{(a,\lambda)\in\mathbb{C}^2: |a|^2+|\lambda|^2<1\right\},$$
which is the unit ball in $\mathbb{C}^2$. Thus, the mapping $\phi:\widetilde\Omega\rightarrow\widetilde\Omega$ given by
$$\phi: (a,\lambda)\mapsto\left(F_1(a,\lambda,\lambda,\lambda^2), \lambda\right)$$
is a proper holomorphic self-mapping of $\widetilde\Omega$. Then by Alexander's Theorem,  there exists $\eta\in\mathbb{T}$ such that
$$\phi(a,\lambda)=(\eta a, \lambda).$$
Hence,
$$F_1(a,\lambda,\lambda,\lambda^2)=\eta a.$$

Thus, the complex Jacobian determinant
$$J_{G}(a,0,0,0)\equiv 1, \quad |a|<1.$$

Since  $G(a, x)$ is holomorphic on $\overline{\mathbb{H}}$, there exists a neighborhood $U \subset \mathbb{E}$ of $ \mathbf{0} = (0,0,0) $ such that the Jacobian determinant satisfies
$$
J_G(a, \mathbf{0}) \neq 0 \quad \text{for all }  (a, x) \in \overline{\mathbb{H}} \cap (\mathbb{C} \times U).
$$
By \eqref{u}, direct computation yields $\frac{\partial^2 u}{\partial x_1 \partial \bar{x}_1}(0,0,0) = 1$. Thus, after shrinking $U$ if necessary, we may assume
\begin{equation}\label{eq3.3}
\frac{\partial^2 u}{\partial x_1 \partial \bar{x}_1}(x_1, x_2, x_3) \neq 0 \quad \text{for } (x_1,x_2,x_3) \in U.
\end{equation}

Thus, the restriction of $G(a, x) = (F_1(a, x), x)$ to the domain $ \mathbb{H} \cap (\mathbb{C} \times U)$ is a biholomorphic mapping. Consequently, for each fixed  $x= (x_1, x_2, x_3) \in U$, define the associated disk
$$
\Omega_{x} = \{ a \in \mathbb{C} : |a|^2 < e^{-u(x)} \}.
$$
Observe that  $F_1(\cdot, x)$ is a biholomorphic self-mapping of $\Omega_{x}$ for every $x \in U$. This implies that
$$
F_1(a, x) = e^{i \theta(x)} \frac{a - F_1^{-1}(0, x)}{1 - \overline{F_1^{-1}(0, x)} \, e^{u(x)} a}.
$$

Define 
$$H(a,x):= e^{-i\theta(x)} \left(1 - \overline{F_1^{-1}(0,x)} \, e^{u(x)} a\right).
$$
By Riemann's removable singularity theorem, $H(a,x)$ is holomorphic on  $\mathbb{H} \cap (\mathbb{C} \times U)$. This implies both
$$e^{i\theta(x)} = H(0,x) \quad \text{and} \quad e^{-i\theta(x)} \overline{F_1^{-1}(0,x)} e^{u(x)} = -\frac{\partial H}{\partial a}(0,x)$$
are holomorphic on $U$. Since $|e^{-i\theta(x)}| = 1$, there exists a real constant $\theta$ such that $e^{-i\theta(x)} \equiv e^{-i\theta}$ throughout $U$.

Thus the function
\[
L(x) := \overline{F_1^{-1}(0,x)}  e^{u(x)}
\]
is holomorphic throughout $U$. Under the assumption that $F_1^{-1}(0,x)$ does not vanish identically, there must exists an open subset $V \subset U$ such that $F_1^{-1}(0,x) \neq 0$ holds for every $x \in V$. Consequently, we deduce that 
\[
u(x) = \ln \left| \frac{L(x)}{F_1^{-1}(0,x)} \right|, \quad x \in V,
\]
which implies that $u$ is pluriharmonic on $V$. This conclusion directly contradicts condition \eqref{eq3.3}. We therefore conclude that $F_1^{-1}(0,x) \equiv 0$ identically on $U$.

This vanishing result leads to the simplified functional form
\[
F_1(a,x) = e^{i\theta} a \quad \text{for} \quad (a,x) \in \mathbb{H} \cap (\mathbb{C} \times U),
\]
and consequently, by the identity theorem, we obtain 
\[
G(a,x) = (e^{i\theta} a,  x) .
\]
Thus, we establish that $f \in \mathrm{Aut}(\mathbb{H})$, which completes the proof of Theorem \ref{Main}.

\vspace{10pt}
\noindent\textbf{Acknowedgments}\quad 
This work was supported by the National Natural Science Foundation of China (No. 12361131577, 12271411, 12201475). Guicong Su was also supported by the Fundamental Research Funds for the Central Universities, Grant No. 104972025KFYjc0095.

\addcontentsline{toc}{section}{References}
\phantomsection
\renewcommand\refname{References}

\end{document}